\documentclass{amsart}
\usepackage[a4paper,twoside,inner=2cm,outer=2cm]{geometry}
\usepackage{amssymb,amsmath,amscd}
\usepackage{tikz}
\usetikzlibrary{matrix,arrows}
\usepackage{graphicx}  
\usepackage{float}  
\usepackage{multirow}  
\usepackage{booktabs}  
\usepackage{subcaption}  
\usepackage[hidelinks]{hyperref}

\usepackage{fancyvrb}  
\usepackage{color}
\usepackage[utf8]{inputenc}

\theoremstyle{definition}
\newtheorem{dfn}{Definition}[section]

\theoremstyle{remark}
\newtheorem{rmk}[dfn]{Remark}
\numberwithin{equation}{subsection}
\theoremstyle{plain}
\newtheorem{thm}[dfn]{Theorem}
\newtheorem{prop}[dfn]{Proposition}
\newtheorem{lem}[dfn]{Lemma}
\newtheorem{cor}[dfn]{Corollary}

\renewcommand{\leq}{\leqslant}
\renewcommand{\geq}{\geqslant}

\newcommand{\Z}{\mathbb{Z}}
\newcommand{\Q}{\mathbb{Q}}
\newcommand{\R}{\mathbb{R}}

\title{Torsions of integral homology and cohomology of real Grassmannians}

\subjclass[2010]{Primary 57T15; Secondary 05A15, 14M15}

\keywords{Real Grassmannian, Integral homology and cohomology, Mod 2 torsion}

\author[He]{\bfseries Chen He}

\email{hechen123@gmail.com}

\begin{document}

\vspace{18mm} \setcounter{page}{1} \thispagestyle{empty}

\begin{abstract}
According to Ehresmann \cite{Eh37}, the torsions of integral homology of real Grassmannian are all of order $2$. In this note, We compute the $\Z_2$-dimensions of torsions in the integral homology and cohomology of real Grassmannian. 
\end{abstract}

\maketitle


\section{Introduction}
\vskip 15pt

The study of homology groups of real Grassmannian $G_k(\R^n)$ was introduced by Ehresmann \cite{Eh37} in terms of Schubert cells $[a_1,a_2,\ldots,a_k],\,0\leq a_1\leq a_2\leq\ldots\leq a_k \leq n-k$ of dimensions $\sum_{i=1}^{k}a_i$.

Noticing the boundary map of those Schubert cells vanishes in $\Z_2$ coefficients, Ehresmann proved 

\begin{thm}[\cite{Eh37} P73]
	The Schubert cells give a basis of $H_*(G_k(\R^n),\Z_2)$.
\end{thm}

Every dimension-$d$ Schubert cell $[a_1,a_2,\ldots,a_k],\,0\leq a_1\leq a_2\leq\ldots\leq a_k \leq n-k,\,\sum_{i=1}^{k}a_i =d$ corresponds to a partition of $d$. 

\begin{dfn}
	Denote $p(M, N; d)$ as the number of partitions of $d$ into at most $M$ non-negative integers, each of size at most $N$. If $d$ is not a non-negative integer, $p(M, N; d)$ is understood to be $0$.
\end{dfn}

\begin{cor}
	The degree-$d$ mod-$2$ Betti number $\mathrm{dim}_{\Z_2}\,H_d(G_k(\R^n),\Z_2)$ is $p(k, n-k; d)$.
\end{cor}

The generating function of $p(k, n-k; d)$ with respect to $d$ is a polynomial of degree $k(n-k)$ and has a nice expression (see \cite{An76} P33 Theorem 3.1).

\begin{dfn}
	Given non-negative integers $k\leq n$, the Gaussian binomial coefficient $\binom{n}{k}_t$ is a polynomial in variable $t$:
	\[
	\frac{\prod_{i=1}^{n}(1-t^i)}{\prod_{i=1}^{k}(1-t^i)\prod_{i=1}^{n-k}(1-t^i)}=\frac{\prod_{i=n-k+1}^{n}(1-t^i)}{\prod_{i=1}^{k}(1-t^i)}
	\]
\end{dfn}

\begin{cor}\label{cor:Poinc2}
	The Poincar\'e polynomial of $H_*(G_k(\R^n),\Z_2)$ is
	\[
	\sum_{d\geq 0}p(k, n-k; d)t^d =\binom{n}{k}_t=\frac{\prod_{i=1}^{n}(1-t^i)}{\prod_{i=1}^{k}(1-t^i)\prod_{i=1}^{n-k}(1-t^i)}
	\]
\end{cor}

For the integral homology of real Grassmannian, Ehresmann showed 
\begin{thm}[\cite{Eh37} P81]\label{thm:Ehresmann2}
	The torsions of $H_*(G_k(\R^n),\Z)$ are all of order $2$.
\end{thm}

Using an explicit boundary map of the Schubert cells in $\Z$ coefficients, Ehresmann also gave the descriptions of the free and torsion parts of $H_*(G_k(\R^n),\Z)$. 

In this short note, we will give the $\Z_2$-dimensions of the torsion part of $H_*(G_k(\R^n),\Z)$ and $H^*(G_k(\R^n),\Z)$ with the help of the understanding of $H^*(G_k(\R^n),\Z_2)$ and $H^*(G_k(\R^n),\Q)$.

\vskip 20pt
\section{Cohomology rings and Poincar\'e polynomial of real Grassmannian}
\vskip 15pt

Dually, we can consider the cohomology groups of real Grassmannian in terms of Schubert cohomology classes. On the other hand, the ring structure of cohomology of real Grassmannian can be nicely given in terms of characteristic classes.

In $\Z_2$ coefficients, the cohomology ring structure and the relations between Schubert classes and Stiefel-Whitney classes were given by Chern for real Grassmannians. 

\begin{thm}[\cite{Ch48} P370 Theorem 2]
	Let $w_1,w_2,\ldots,w_k$ be the Stiefel-Whitney classes of the universal bundle on $G_k(\R^n)$, and $\bar{w}_1,\bar{w}_2,\ldots,\bar{w}_{n-k}$ be those of the complementary bundle, then
	\[
	H^*(G_k(\R^n),\Z_2)=\frac{\Z_2[w_1,w_2,\ldots,w_k;\bar{w}_1,\bar{w}_2,\ldots,\bar{w}_{n-k}]}{(1+w_1+\cdots+w_k)(1+\bar{w}_1+\cdots+\bar{w}_{n-k})=1}
	\]
\end{thm}

\begin{rmk}\label{rmk:BorelLemma}
	Using a Lemma of Borel (\cite{Bo53} P190 Lemma 26.2, also see \cite{BT82} P294-297), the Poincar\'e polynomial of the above quotient of polynomial ring can be calculated to be the Gaussian binomial coefficient $\binom{n}{k}_t$, same as we see in Corollary \ref{cor:Poinc2}.
\end{rmk}

In rational coefficients, the cohomology ring of even-dimensional real Grassmannian follows from a general theorem of Leray and Borel.

\begin{thm}[\cite{Le49} P1902 Theorem a, \cite{Bo53} P191 Theorem 26.1]
	For compact connected Lie group $G$ and its closed subgroup $H$ containing the same maximal torus $T$, denote the Weyl groups of $G,H$ as $W_G,W_H$, and the symmetric algebra of the dual Lie algebra $\mathfrak{t}^*$ as $\mathbb{S}\mathfrak{t}^*$, then
	\[
	H^*(G/H,\Q)=(\mathbb{S}\mathfrak{t}^*)^{W_H}/(\mathbb{S}\mathfrak{t}^*)_+^{W_G}
	\]
	where $(\mathbb{S}\mathfrak{t}^*)^{W_H}$ is the set of $W_H$ invariants, and $(\mathbb{S}\mathfrak{t}^*)_+^{W_G}$ is the set of positive-degree $W_G$ invariants. 
\end{thm}

For any real number $a$, denote $[a]$ as it integer part. We can apply the above Leray-Borel theorem to even-dimensional real Grassmannian. 

\begin{thm}
	Suppose the dimension $k(n-k)$ of $G_{k}(\R^{n})$ is even. Let $p_1,p_2,\ldots,p_{[\frac{k}{2}]}$ be the Pontryagin classes of the universal bundle on $G_k(\R^n)$, and $\bar{p}_1,\bar{p}_2,\ldots,\bar{p}_{[\frac{n-k}{2}]}$ be those of the complementary bundle, then
	\begin{align*}
	H^*({G}_{k}(\R^{n}),\Q) =\frac{\Q[p_1,p_2,\ldots,p_{[\frac{k}{2}]};\bar{p}_1,\bar{p}_2,\ldots,\bar{p}_{[\frac{n-k}{2}]}]}{(1+p_1+\cdots+p_{[\frac{k}{2}]})(1+\bar{p}_1+\cdots+\bar{p}_{[\frac{n-k}{2}]})=1}
	\end{align*}
\end{thm}

As in Remark \ref{rmk:BorelLemma}, the Poincar\'e polynomial of the above quotient of polynomial ring can be calculated. Also note that each Pontryagin class $p_i$ is of degree $4i$.

\begin{cor}
	For an even-dimensional $G_{k}(\R^{n})$, the Poincar\'e polynomial of $H^*(G_k(\R^n),\Q)$ is
	\[
	\binom{[\frac{k}{2}]+[\frac{n-k}{2}]}{[\frac{k}{2}]}_{t^4} =\frac{\prod_{i=1}^{[\frac{k}{2}]+[\frac{n-k}{2}]}(1-t^{4i})}{\prod_{i=1}^{[\frac{k}{2}]}(1-t^{4i})\prod_{i=1}^{[\frac{n-k}{2}]}(1-t^{4i})}
	\]
\end{cor} 

Using the relation between Gaussian binomial coefficients and restricted partition numbers (Corollary \ref{cor:Poinc2} or \cite{An76} P33 Theorem 3.1), we get the finite-term expansion
\[
\frac{\prod_{i=1}^{[\frac{k}{2}]+[\frac{n-k}{2}]}(1-t^{4i})}{\prod_{i=1}^{[\frac{k}{2}]}(1-t^{4i})\prod_{i=1}^{[\frac{n-k}{2}]}(1-t^{4i})}=\sum_{d\geq 0}p\Big(\big[\frac{k}{2}\big], \big[\frac{n-k}{2}\big]; d\Big)t^{4d}
\]

\begin{cor}
	For an even-dimensional $G_{k}(\R^{n})$, the degree-$d$ rational Betti number 
	\[
	\mathrm{dim}_\Q\,H^d(G_k(\R^n),\Q)=p\Big(\big[\frac{k}{2}\big], \big[\frac{n-k}{2}\big]; \frac{d}{4}\Big)
	\]
	which is nonzero only when $d$ is divisible by $4$.
\end{cor}

The computation of rational cohomology of odd-dimensional real Grassmannian was due to Takeuchi.

\begin{thm}[\cite{Ta62} P320]
	Suppose the dimension $k(n-k)$ of $G_{k}(\R^{n})$ is odd, i.e. $k$ odd, $n$ even, then there is an odd-degree element $r\in H^{n-1}(G_{k}(\R^{n}),\Q)$ such that
	\begin{align*}
	H^*({G}_{k}(\R^{n}),\Q) =\frac{\Q[p_1,p_2,\ldots,p_{[\frac{k}{2}]};\bar{p}_1,\bar{p}_2,\ldots,\bar{p}_{[\frac{n-k}{2}]};r]}{(1+p_1+\cdots+p_{[\frac{k}{2}]})(1+\bar{p}_1+\cdots+\bar{p}_{[\frac{n-k}{2}]})=1; r^2=0}
	\end{align*}
\end{thm}

\begin{cor}
	For an odd-dimensional $G_{k}(\R^{n})$, the Poincar\'e polynomial of $H^*(G_k(\R^n),\Q)$ is
	\[
	(1+t^{n-1})\binom{[\frac{k}{2}]+[\frac{n-k}{2}]}{[\frac{k}{2}]}_{t^4} =(1+t^{n-1})\frac{\prod_{i=1}^{[\frac{k}{2}]+[\frac{n-k}{2}]}(1-t^{4i})}{\prod_{i=1}^{[\frac{k}{2}]}(1-t^{4i})\prod_{i=1}^{[\frac{n-k}{2}]}(1-t^{4i})}
	\]
	with finite-term expansion
	\[
	\sum_{d\geq 0}p\Big(\big[\frac{k}{2}\big], \big[\frac{n-k}{2}\big]; d\Big)t^{4d}+\sum_{d\geq 0}p\Big(\big[\frac{k}{2}\big], \big[\frac{n-k}{2}\big]; d\Big)t^{4d+n-1}
	\]
	And the degree-$d$ rational Betti number 
	\[
	\mathrm{dim}_\Q\,H^d(G_k(\R^n),\Q)= p\Big(\big[\frac{k}{2}\big], \big[\frac{n-k}{2}\big]; \frac{d}{4}\Big)+p\Big(\big[\frac{k}{2}\big], \big[\frac{n-k}{2}\big]; \frac{d-n+1}{4}\Big)
	\]
	which is nonzero only when $d$ or $d-n+1$ is divisible by $4$.
\end{cor}

\begin{rmk}
	Casian and Kodama (\cite{CK} P11 Theorem 5.1, P12 Conjecture 6.1) used combinatorial argument related with Young diagrams to calculate the Poincar\'e polynomial of rational(integral) cohomology of real Grassmannian, and stated the rational ring structure in terms of characteristic classes.
\end{rmk}

\vskip 20pt
\section{Generating function of dimensions of torsions}
\vskip 15pt
With the $\Z_2$ and rational cohomology understood, we can work out the integral homology and cohomology groups of real Grassmannian using the following simple lemmas.

\begin{lem}
	Assume the torsions of the integral homology of a compact manifold $M$ are all of the order of a prime $p$, i.e. $H_d(M,\Z)\cong\Z^{{FB}_d}\oplus\Z_p^{{TB}_d}$ where ${FB}_d = \dim_\Q H_*(M,\Q) = \dim_\Q H^*(M,\Q)$. Then we have
	\begin{align*}
	H^d(M,\Z)&\cong\Z^{{FB}_d}\oplus\Z_p^{{TB}_{d-1}}\\
	H_d(M,\Z_p)&\cong H^d(M,\Z_p) \cong \Z_p^{{FB}_d+{TB}_{d-1}+{TB}_d}
	\end{align*}
\end{lem}
\begin{proof}
	This follows immediately from the universal coefficient theorems of homology and cohomology in $\Z,\Z_p,\Q$ coefficients.
\end{proof}

Denote $B_d$ as the integer $\mathrm{dim}_{\Z_p}\,H_d(M,\Z_p)=\mathrm{dim}_{\Z_p}\, H^d(M,\Z_p)={FB}_d+{TB}_{d-1}+{TB}_d$. Let's consider the generating functions
\begin{align*}
P&=\sum_{d\geq 0}B_d t^d & FP&=\sum_{d\geq 0}{FB}_d t^d\\
\underline{TP}&=\sum_{d\geq 0}{TB}_d t^d & \overline{TP}&=\sum_{d\geq 0}{TB}_{d-1} t^d = t\,\underline{TP}
\end{align*}
where $P,\,FP$ are the Poincar\'e polynomials of the $\Z_p$ and rational homology(cohomology) of $M$, and $\underline{TP},\,\overline{TP}$ are the generating functions of the dimensions of torsions in integral homology and cohomology of $M$.

\begin{lem}\label{lem:PoincTorsion}
	Under the same assumption of the previous Lemma, we have
	\begin{align*}
	\underline{TP}&=\frac{P-FP}{1+t}\\
	\overline{TP}&=\frac{P-FP}{1+t^{-1}}=t\,\frac{P-FP}{1+t}
	\end{align*}
\end{lem}
\begin{proof}
	Since $B_d={FB}_d+{TB}_{d-1}+{TB}_d$, then
	\begin{align*}
	P &=\sum_{d\geq 0}B_d t^d=\sum_{d\geq 0}{FB}_d t^d+\sum_{d\geq 0}{TB}_{d-1} t^d+\sum_{d\geq 0}{TB}_d t^d\\
	&=FP+\overline{TP}+\underline{TP}=FP+t\,\underline{TP}+\underline{TP}\\
	&=FP+(1+t)\,\underline{TP}
	\end{align*}
	Similarly, we have $P=FP+(1+t^{-1})\,\overline{TP}$.
\end{proof}
\begin{rmk}
	The difference $P-FP$ is divisible by $1+t$, because both $P(-1)$ and $FP(-1)$ give the Euler characteristic number of $M$.
\end{rmk}

We can also express the dimension ${TB}_d$ of torsions directly.
\begin{lem}\label{lem:BettiTorsion}
	Expanding the both sides of the equalities in the previous Lemma, we get
	\begin{align*}
	{TB}_d  = (B_d-{FB}_d)-(B_{d-1}-{FB}_{d-1})+(B_{d-2}-{FB}_{d-2})\pm\cdots
	\end{align*}
\end{lem}

By Theorem \ref{thm:Ehresmann2}, the torsions of integral homology of real Grassmannian are all of order $2$. Hence we can apply the above lemmas together with the understanding of the $\Z_2$ and rational cohomology of real Grassmannian.

\begin{prop}
	For a real Grassmannian $G_{k}(\R^{n})$,
	\begin{itemize}
		\item if its dimension $k(n-k)$ is even, then the generating function of dimensions of the free part in $H_*(G_{k}(\R^{n}),\Z)$ and $H^*(G_{k}(\R^{n}),\Z)$ is 
		\[
		\binom{[\frac{k}{2}]+[\frac{n-k}{2}]}{[\frac{k}{2}]}_{t^4}
		\]
		The generating functions of dimensions of the $\Z_2$-torsions in $H_*(G_{k}(\R^{n}),\Z)$ and $H^*(G_{k}(\R^{n}),\Z)$ are respectively
		\begin{align*}
		\frac{1}{1+t}\bigg[\binom{n}{k}_t-\binom{[\frac{k}{2}]+[\frac{n-k}{2}]}{[\frac{k}{2}]}_{t^4}\bigg] && \frac{t}{1+t}\bigg[\binom{n}{k}_t-\binom{[\frac{k}{2}]+[\frac{n-k}{2}]}{[\frac{k}{2}]}_{t^4}\bigg]
		\end{align*}
		\item if its dimension $k(n-k)$ is odd, then the generating function of dimensions of the free part in $H_*(G_{k}(\R^{n}),\Z)$ and $H^*(G_{k}(\R^{n}),\Z)$ is 
		\[
		(1+t^{n-1})\binom{[\frac{k}{2}]+[\frac{n-k}{2}]}{[\frac{k}{2}]}_{t^4}
		\]
		The generating functions of dimensions of the $\Z_2$-torsions in $H_*(G_{k}(\R^{n}),\Z)$ and $H^*(G_{k}(\R^{n}),\Z)$ are respectively
		\begin{align*}
		\frac{1}{1+t}\bigg[\binom{n}{k}_t-(1+t^{n-1})\binom{[\frac{k}{2}]+[\frac{n-k}{2}]}{[\frac{k}{2}]}_{t^4}\bigg] && \frac{t}{1+t}\bigg[\binom{n}{k}_t-(1+t^{n-1})\binom{[\frac{k}{2}]+[\frac{n-k}{2}]}{[\frac{k}{2}]}_{t^4}\bigg]
		\end{align*}
	\end{itemize}
	Moreover, the dimensions of torsions can be expressed in terms of restricted partition numbers.
\end{prop}

\begin{rmk}
	Using Ehresmann's explicit boundary map of Schubert cells, Junkind (\cite{Ju79} P36 Table IV) computed the integral homology groups of $G_{k}(\R^{n})$ for $n\leq 7$.
\end{rmk}

The above formulas can be implemented in computer algebra systems. Here we give Python 2.7 codes using the Sympy module.

\input{codes.tex}

\vskip 20pt
\section{Acknowledgment}
\vskip 15pt
The author thanks Luis Casian and Yuji Kodama for the invitation to talk at the seminar of Geometry, Combinatorics, and Integrable Systems at OSU. The author also thanks David Anderson for asking the question of torsions that motivates this note.

\vskip 20pt
\bibliographystyle{amsalpha}

\end{document}